\documentclass[10pt,a4paper]{amsart}
\usepackage{amsfonts}
\usepackage[latin9]{inputenc}
\usepackage{amsmath}
\usepackage{amssymb}
\usepackage{breqn}
\usepackage{url}
\usepackage{graphicx}
\usepackage[pdfstartview=FitH]{hyperref}
\usepackage{amsthm}
\usepackage{hyperref}
\usepackage{url}
\usepackage{enumitem}

\newcommand{\im}{\operatorname{Im}}

\makeatletter

\begin{document}
\newtheorem{theorem}{Theorem}[section]
\newtheorem{lemma}[theorem]{Lemma}
\newtheorem{definition}[theorem]{Definition}
\newtheorem{claim}[theorem]{Claim}
\newtheorem{example}[theorem]{Example}
\newtheorem{remark}[theorem]{Remark}
\newtheorem{proposition}[theorem]{Proposition}
\newtheorem{corollary}[theorem]{Corollary}
\newtheorem{observation}[theorem]{Observation}
\newcommand{\subscript}[2]{$#1 _ #2$}
\newtheorem*{theorem*}{Theorem}

\author{Izhar Oppenheim}
\address{Department of Mathematics, Ben-Gurion University of the Negev, Be'er Sheva 84105, Israel} 
\email{izharo@bgu.ac.il}
\title[Angle criteria for uniform convergence]{Angle criteria for uniform convergence of averaged projections and cyclic or random products of projections}
\maketitle
\begin{abstract}
We apply a new notion of angle between projections to deduce criteria for uniform convergence results of the alternating projections method under several different settings: averaged projections, cyclic products, quasi-periodic products and random products.
\end{abstract}

\section{Introduction}

Given a set of orthogonal projections $\lbrace P_1,...,P_n \rbrace$ in a Hilbert space $\mathcal{H}$ there is a well-known iterative method called the alternating projections method for approximating the orthogonal projection of some $v \in \mathcal{H}$ on $\im (P_1) \cap ... \cap \im (P_n)$. There are several variations of this method, but perhaps its simplest form is to consider cyclic products of $P_1,...,P_n$, i.e., to look at the sequence  $(P_n ... P_1)^i$. In \cite{Halperin}, Halperin showed that $(P_n ... P_1)^i$ converges to the projection on $\im (P_1) \cap ... \cap \im (P_n)$ in the strong operator topology, i.e, he showed that for every $v \in \mathcal{H}$ 
$$\lim_{i \rightarrow \infty} \Vert (P_n ... P_1)^i v - P_{1,...,n} v \Vert = 0,$$
where $P_{1,...,n}$ is the orthogonal projection on $\im (P_1) \cap ... \cap \im (P_n)$ (when $n=2$ this result is due to von Neumann \cite{vNeumann}). 

Extending Halperin's result, one can consider non-cyclic sequences of the projections $\lbrace P_1,...,P_n \rbrace$. This can be described as follows - take $\tau \in \lbrace 1,...,n \rbrace^{\mathbb{N}}$ and consider the sequence $...P_{\tau (i)} ... P_{\tau (1)}$. Sakai \cite{S} showed that if $\tau$ is quasi-periodic (see definition below) then the result of Halperin holds. Amemiya and And{\^o} \cite{AA} showed that if for every $1 \leq k \leq n$, the set $\lbrace j : \tau (j) = k \rbrace$ is infinite, then $...P_{\tau (i)} ... P_{\tau (1)}$ converges to $P_{1,...,n}$ in the weak-operator topology. This result was generalized by several authors, who gave sufficient conditions for stronger forms of convergence in the Hilbert or Banach settings (see \cite{Dye}, \cite{DR}, \cite{DR2}, \cite{DKR}, \cite{DKLR}). On the other hand, it was shown by Kopeck{\'a} and M{\"u}ller in \cite{KM} (improving on a similar result of Paszkiewicz \cite{Pas}) that in general even in the case of $3$ orthogonal projections, $...P_{\tau (i)} ... P_{\tau (1)}$ need not converge in the strong operator norm (this result was later strengthened in \cite{KP}).

Another of the method approximating the orthogonal projection on $\im (P_1) \cap ... \cap \im (P_n)$ is averaged projections. In this methods, one fixes positive constants $\alpha_k >0, k=1,...,n$ such that $\sum_k \alpha_k =1$ and considers the convergence of $\left( \sum_{k} \alpha_k P_k \right)^i$ as $i$ tends to infinity. In \cite{Lapidus}, Lapidus showed that this method converges to a projection on $\im (P_1) \cap ... \cap \im (P_n)$ in the strong operator topology. This result was extended to various types Banach spaces (such as uniformly convex spaces) in \cite{Reich} and \cite{BL} (under the assumption that all the projections are of norm $1$).

In this paper, we will give sufficient criteria for uniform convergence in the case of averaged projections and in the cases of cyclic, quasi-periodic and random products for projections on Banach spaces - see Theorem \ref{averaged projection theorem}, Proposition \ref{prop cyclic}, Theorem \ref{thm quasi-periodic} and Theorem \ref{thm random} below. Our criteria uses a new notion of angles between projections that we originally introduced in \cite{ORobust} (and later developed in \cite{OBanachCoho}) in order to study representations of groups on Banach spaces. The basic idea behind all our criteria is that if the angle between any two projections is large enough (with respect to other parameters, such the norms of the projections) then uniform convergence is achieved.
 
Since various authors used conditions regarding angles in the Hilbert setting to deduce similar uniform convergence results (see for instance \cite{BadeaGrivauxM}, \cite{PRZ} and \cite{PRZ2}), we found it necessary to point out the new ideas in this paper:
\begin{itemize}
\item Unlike previous work, we generalize the notion of angles to a notion of angle between projections in Banach setting. Two things are new here - the notion of angles between projections (as opposed to angle between subspaces) and the generalization of this notion of an angle to Banach spaces. Let us remark that while there are other definitions of angles between subspaces in Banach spaces (see for instance \cite{Ostrovski}[section 3]), as far as we know, these definitions are not applicable to ensure convergence of the alternating projections or the average projections methods. 
\item Our work does not require the projections to be of norm $1$ (although the norm should be sufficiently close to $1$).
\item Our criteria refer only to angles between couples of projections (i.e., we ask that for each $i,j$, the angle between $P_i$ and $P_j$ will be large), while previous articles defined angles between the $n$-tuple of subspaces. This point was important in applying our results in concrete examples. Namely, in \cite{ORobust} and \cite{OBanachCoho} we had concrete examples of projection (related to examples of groups) such that the angles between projection were effectively bounded for several classes of Banach spaces (e.g., in $L^p$ spaces).
\end{itemize} 

\textbf{Acknowledgements.} I would like to thank Piotr Nowak for pointing out a mistake in an early draft of this paper and Simeon Reich for his many useful suggestions.

\section{Angle between projections - definition and uniform convergence of averaged projections}  
Let $X$ be a Banach space. Recall that a projection $P$ is a bounded operator $P \in \mathcal{B} (X)$ such $P^2 =P$.  Note that $\Vert P \Vert \geq 1$ if $P \neq 0$. 


\begin{definition}[Angle between projections]
\label{angle between projections definition}
Let $X$ be a Banach space and let $P_1, P_2$ be projections. Assume that there is a projection $P_{1,2}$ on $\im (P_1) \cap \im (P_2)$ such that $P_{1,2} P_1 = P_{1,2}$ and $P_{1,2} P_2 = P_{1,2}$ and define 
$$\cos (\angle (P_1,P_2)) = \max \left\lbrace \Vert P_1 (P_2 - P_{1,2} ) \Vert, \Vert P_2 (P_1 - P_{1,2} ) \Vert  \right\rbrace.$$ 
\end{definition}

A few remarks regarding this definition are in order:

\begin{remark}
In the above definition, we are actually defining the ``cosine'' of the angle. This is a little misleading, because it could be that in some cases  $\cos (\angle (P_1,P_2)) > 1$. 
\end{remark}

\begin{remark}
We note that in the case where $X$ is a Hilbert space and $P_1,P_2$ are orthogonal projections on $V_1,V_2$, the orthogonal projection $P_{1,2}$ on  $V_1 \cap V_2$ will always fulfill $P_{1,2} P_1 = P_{1,2}$ and $P_{1,2} P_2 = P_{1,2}$. Also, in this case, $\cos (\angle (P_1,P_2))$ will be equal to the Friedrichs angle between $V_1$ and $V_2$ defined by
$$\cos (\angle (V_1,V_2))= \sup \lbrace \vert \langle u,v \rangle \vert : \Vert u \Vert \leq 1, \Vert v \Vert \leq 1, u \in V_1 \cap (V_1 \cap V_2)^\perp, v \in V_2 \rbrace.$$
\end{remark} 

\begin{remark}
The conditions $P_{1,2} P_1 = P_{1,2}$ and $P_{1,2} P_2 = P_{1,2}$ in the above definition can be understood as follows - in order to define the Friedrichs angle between two subspaces in a Hilbert space, we needed to quotient out their intersection. In Banach spaces, taking a quotient is a non linear operation in the sense that there is no canonical linear projection on the quotient. The condition $P_{1,2} P_1 = P_{1,2}$ replaces the notion of quotient by providing a linear projection $P_1-P_{1,2}$ on the subspace of vectors fixed by $P_1$ but not fixed by $P_{1,2}$ (except for the $0$-vector which is of course fixed by both). The condition $P_{1,2} P_2 = P_{1,2}$ serves the same purpose making $P_2-P_{1,2}$ into a similar projection.
\end{remark}

\begin{remark}
The above definition of angle might seem hard to compute or even bound. However, the reader can find a multitude of examples in \cite{ORobust} and \cite{OBanachCoho} in which such bounds are computed in a large class of Banach spaces. The main theme of these bounds is that one can use bounds of the Friedrichs angle between two orthogonal projections in a Hilbert space in order to give bounds for the cosine of the angle between ``similar'' projections in Banach spaces, given that the Banach spaces are ``close enough'' to some Hilbert space. 
\end{remark}

Next, we'll prove an ``angle criterion'' for uniform convergence of the method of averaged projections. A basic version of this result has already appeared in \cite{ORobust}. Namely, in \cite{ORobust}[Theorem 3.12], we gave an ``angle criterion'' for uniform convergence of $( \frac{P_1+...+P_n}{n} )^i$ to a projection on $\im (P_1) \cap ... \cap \im (P_n)$, given that the angle between every couple $P_i$ and $P_j$ is small enough and that the norm of all the $P_i$ is sufficiently close to $1$. Below, we'll deal with the more general case of the convergence of any convex combination of $P_1,...,P_n$.  We note that the angle condition we derive below is stronger than \cite{ORobust}[Theorem 3.12] even in the simplified case of $( \frac{P_1+...+P_n}{n} )^i$, since it allows interplay between different angles, i.e., not all the angles have to be large (a few small angles can be compensated by many large angles). 

\begin{theorem}
\label{averaged projection theorem}
Let $X$ be a Banach space and let $P_1,...,P_n$ be projections in $X$ ($n \geq 2$). Assume that for every $1 \leq j_1 < j_2 \leq n$, there is a projection $P_{j_1,j_2}$ on $\im (P_{j_1}) \cap \im (P_{j_2})$, such that $P_{j_1,j_2} P_{j_1} = P_{j_1,j_2}, P_{j_1,j_2} P_{j_2} = P_{j_1,j_2}$. Let $\alpha_1,...,\alpha_n$ be positive constants such that $\sum_k \alpha_k =1$. Denote $T = \sum_k \alpha_k P_k$. Assume there is a constant 
$$\beta < \min \left\lbrace \dfrac{1}{1-\alpha_1},..., \dfrac{1}{1-\alpha_n} \right\rbrace,$$
such that 
$$ \max \lbrace \Vert P_1 \Vert,..., \Vert P_n \Vert \rbrace \leq \beta .$$
Assume further that for every $1 \leq j_1, j_2 \leq n$, $\cos (\angle (P_{j_1},P_{j_2})) <1$ and that 
$$r=\min_{1 \leq j \leq n}   \left( (1-\alpha_j)\beta +\sum_{k=1, k \neq j}^n \left(\dfrac{(\cos (\angle (P_j,P_k))+ \beta+\beta^2) \cos (\angle (P_j,P_k))}{1-\cos (\angle (P_j,P_k))} 2 \alpha_k \right) \right) <1.$$
Then there is a projection $T^\infty$ on $\bigcap_{j=1}^n \im (P_j)$, such that for every $i$, 
$$\Vert T^i - T^\infty \Vert \leq Cr^i$$ 
with $C =\frac{1+\beta}{1-r}$.
\end{theorem}

\begin{remark}
Note that in the theorem above the first condition
$$\beta < \min \left\lbrace \dfrac{1}{1-\alpha_1},..., \dfrac{1}{1-\alpha_n} \right\rbrace$$
is necessary for the fulfilment of the second condition  $r<1$.
\end{remark}

In order to prove the above theorem, we will need the following lemma:
\begin{lemma}
\label{p1p2-p2p1 lemma}
Let $P_1,P_2$ be projections in a Banach space $X$, such that there is a projection $P_{1,2}$ on $\im (P_1) \cap \im (P_2)$ with $P_{1,2} P_1 = P_{1,2}, P_{1,2} P_2 = P_{1,2}$. Assume that $\cos (\angle (P_1,P_2))<1$, and that $\beta$ is a constant such that $\Vert P_1 \Vert \leq \beta$ and  $\Vert P_2 \Vert \leq \beta$.
Then for any operator $S \in B(X)$, we have that
\begin{dmath*}
\Vert (P_1 P_2 - P_2 P_1)S \Vert \leq \\
 \dfrac{(\beta + \beta^2 + \cos (\angle (P_1,P_2)))  \cos (\angle (P_1,P_2))}{1-\cos (\angle (P_1,P_2))} \left( \Vert (I-P_1)S \Vert + \Vert (I-P_2)S \Vert \right)
\end{dmath*}

\end{lemma}

\begin{proof}
Let $S \in B(X)$ be some operator. We start by observing that
\begin{dmath*}
\Vert (P_1 P_2 - P_2 P_1)S \Vert \leq \\
 \Vert (P_1 P_2 - P_{1,2})S \Vert + \Vert (P_2 P_1 - P_{1,2})S \Vert = \Vert P_1 ( P_2 - P_{1,2})^2 S \Vert + \Vert P_2 ( P_1 - P_{1,2})^2 S \Vert \leq \\
  \cos (\angle (P_1,P_2)) \left( \Vert (P_1 - P_{1,2})S \Vert + \Vert (P_2 - P_{1,2})S \Vert \right).
\end{dmath*}
In order to complete the proof, we will show that
$$ \Vert (P_1 - P_{1,2})S \Vert + \Vert (P_2 - P_{1,2})S \Vert \leq \dfrac{\beta + \beta^2 + \cos (\angle (P_1,P_2))}{1-\cos (\angle (P_1,P_2))} (\Vert (I-P_1)S \Vert + \Vert (I-P_2)S \Vert).$$
First, note that by the definition of $\cos (\angle (P_1,P_2))$, we have that 
$$\Vert P_{1,2} \Vert \leq \Vert P_1 P_2 \Vert + \cos (\angle (P_1,P_2)) \leq \beta^2 +  \cos (\angle (P_1,P_2)).$$
Observe that
\begin{dmath*}
 \Vert (P_1 - P_{1,2})S \Vert \leq  \Vert (P_1 - P_{1,2}) P_2 S \Vert +  \Vert  (P_1 - P_{1,2})(I-P_2)S \Vert = \\
 \Vert (P_1 P_2 - P_{1,2}) S \Vert +  \Vert  (P_1 - P_{1,2}) (I-P_2)S \Vert = \\
 \Vert P_1 (P_2 - P_{1,2})^2 S \Vert +  \Vert  (P_1 - P_{1,2}) (I-P_2)S \Vert  \leq \\
 \cos (\angle (P_1,P_2)) \Vert  (P_2 - P_{1,2})S \Vert + (\Vert P_1 \Vert + \Vert P_{1,2} \Vert) \Vert (I-P_2)S \Vert \leq \\
 \cos (\angle (P_1,P_2)) \Vert  (P_2 - P_{1,2})S \Vert + (\beta + \beta^2 + \cos (\angle (P_1,P_2)) ) \Vert (I-P_2)S \Vert,
\end{dmath*}
where the last inequality is due to the bound on $\Vert P_{1,2} \Vert$ noted above. 

By similar computations
\begin{dmath*}
\Vert (P_2 - P_{1,2})S \Vert \leq  \\
 \cos (\angle (P_1,P_2)) \Vert  (P_1 - P_{1,2})S \Vert + (\beta + \beta^2 + \cos (\angle (P_1,P_2)) ) \Vert (I-P_1)S \Vert.
\end{dmath*}

Summing these two inequalities yields
\begin{dmath*}
 \Vert (P_1 - P_{1,2})S \Vert + \Vert (P_2 - P_{1,2})S \Vert \leq \\
 \cos (\angle (P_1,P_2)) (\Vert  (P_1 - P_{1,2})S \Vert+\Vert  (P_2 - P_{1,2})S \Vert) + \\ (\beta + \beta^2 + \cos (\angle (P_1,P_2)) ) (\Vert (I-P_1)S \Vert + \Vert (I-P_2)S \Vert).
\end{dmath*}
Therefore
$$ \Vert (P_1 - P_{1,2})S \Vert + \Vert (P_2 - P_{1,2})S \Vert \leq \dfrac{\beta + \beta^2 + \cos (\angle (P_1,P_2))}{1-\cos (\angle (P_1,P_2))} (\Vert (I-P_1)S \Vert + \Vert (I-P_2)S \Vert),$$
as needed.
\end{proof}

Next, we can prove Theorem \ref{averaged projection theorem}:
\begin{proof}
We start by defining an energy function $E: B(X) \rightarrow \mathbb{R}_{\geq 0}$ as
$$E(S) = \sum_{j=1}^n \alpha_j \Vert (I-P_j) S \Vert.$$
\textbf{Step 1:} We will show that for $T = \sum_k \alpha_k P_k$, we have for every $S \in B(X)$ that $E (TS) \leq r E(S)$, where $r$ is the constant defined as in Theorem \ref{averaged projection theorem}. Fix some $S \in B(X)$, then
\begin{dmath*}
E(TS) = \sum_{j=1}^n \alpha_j \left\Vert (I-P_j) \left( \sum_{k=1}^n \alpha_k P_k \right) S \right\Vert = 
\sum_{j=1}^n \alpha_j \left\Vert  \left( \sum_{k=1,k\neq j}^n \alpha_k (P_k-P_j P_k) \right) S \right\Vert = 
 \sum_{j=1}^n \alpha_j \left\Vert  \left( \sum_{k=1,k\neq j}^n \alpha_k (P_k-P_k P_j) + \alpha_k (P_k P_j - P_j P_k) \right) S \right\Vert = 
\sum_{j=1}^n \alpha_j \left\Vert  \left( \sum_{k=1,k\neq j}^n \alpha_k P_k (I-P_j) + \alpha_k (P_k P_j - P_j P_k) \right) S \right\Vert \leq 
{\left( \sum_{j=1}^n \alpha_j \sum_{k=1,k\neq j}^n \alpha_k \Vert P_k \Vert \Vert (I-P_j) S \Vert  \right)} +   {\left( \sum_{k,j=1, k \neq j}^n \alpha_j \alpha_k \Vert (P_k P_j - P_j P_k)  S  \Vert \right)}.
\end{dmath*}
We'll deal with each of the summands above separately. First, we notice that
\begin{dmath}
\label{first summand}
\sum_{j=1}^n \alpha_j \sum_{k=1,k\neq j}^n \alpha_k \Vert P_k \Vert \Vert (I-P_j) S \Vert  =\sum_{j=1}^n \alpha_j \Vert (I-P_j) S \Vert  \sum_{k=1,k\neq j}^n \alpha_k \beta = \\
 \sum_{j=1}^n (1-\alpha_j)\beta \alpha_j \Vert (I-P_j) S \Vert.
\end{dmath}

Second, we apply Lemma \ref{p1p2-p2p1 lemma} to the second summand
\begin{dmath}
\label{second summand}
\sum_{k,j=1, k \neq j}^n \alpha_j \alpha_k \Vert (P_k P_j - P_j P_k)  S  \Vert \leq \\ 
\sum_{k,j=1, k \neq j}^n \alpha_j \alpha_k  \dfrac{(\beta + \beta^2 + \cos (\angle (P_j,P_k)))  \cos (\angle (P_j,P_k))}{1-\cos (\angle (P_j,P_k))} {\left( \Vert (I-P_j)S \Vert + \Vert (I-P_k)S \Vert \right)} = \\
\sum_{j=1}^n \alpha_j \Vert (I-P_j)S \Vert \sum_{k=1, k \neq j}^n \dfrac{(\beta + \beta^2 + \cos (\angle (P_j,P_k)))  \cos (\angle (P_j,P_k))}{1-\cos (\angle (P_j,P_k))} 2 \alpha_k.
\end{dmath}
Combining \eqref{first summand} and \eqref{second summand} yields
\begin{dmath*}
E(TS) \leq \\
\sum_{j=1}^n \alpha_j \Vert (I-P_j)S \Vert \left( (1-\alpha_j)\beta + \sum_{k=1, k \neq j}^n \dfrac{(\beta + \beta^2 + \cos (\angle (P_j,P_k)))  \cos (\angle (P_j,P_k))}{1-\cos (\angle (P_j,P_k))} 2 \alpha_k \right) \leq 
r E(S).
\end{dmath*}
\textbf{Step 2:} We'll show that $T^i$ is a Cauchy sequence and therefore convergences. Indeed, let $i \geq 1$, then 
\begin{dmath*}
\Vert T^{i+1} - T^i \Vert = \left\Vert \left( \sum_{j=1}^n \alpha_j P_j - I \right) T^i \right\Vert \leq \sum_{j=1}^n \alpha_j \Vert (I-P_j) T^i \Vert = E(T^i) \leq r^i E(I),
\end{dmath*}
where the last inequality is due to step 1. Recall that we assumed that $r<1$ and note that $E(I)$ is a constant and therefore $T^i$ is indeed a Cauchy sequence. Note that 
$$E(I) = \sum_{j=1}^n \alpha_j \Vert I-P_j \Vert \leq \sum_{j=1}^n \alpha_j (1+\Vert P_j \Vert) \leq 1+ \beta.$$
Denote by $T^\infty$ the limit of $T^i$, then for every $i$ we have that 
$$\Vert T^\infty - T^i \Vert \leq \sum_{k=i}^\infty E(I) r^k = \dfrac{E(I)}{1-r} r^i \leq \dfrac{1+\beta}{1-r} r^i. $$
\textbf{Step 3:} We'll show that $T^\infty$ is a projection on $\im (P_1) \cap ... \cap \im (P_n)$. $T^\infty$ is the limit of $T^i$ and therefore $(T^\infty)^2=T^\infty$, i.e., $T^\infty$ is a projection. We are left to show that $\im (P_1) \cap ... \cap \im (P_n) = \im (T^\infty)$. For each $v \in \im (P_1) \cap ... \cap \im (P_n)$, we have that $P_j v = v$ for every $1 \leq j \leq n$ and therefore $Tv = v$ which implies that $T^\infty v = v$, i.e., $\im (P_1) \cap ... \cap \im (P_n) \subseteq \im (T^\infty)$. On the other hand, $T T^\infty = T^\infty$ and therefore, by step 1, we have that $E(T^\infty) = E(T T^\infty) \leq r E(T^\infty)$, which implies that $E(T^\infty)=0$. This yields $(I-P_j)T^\infty =0$ for every $j$ and therefore $\im (T^\infty) \subseteq \im (P_1) \cap ... \cap \im (P_n)$.
\end{proof}

As a corollary we can deduce a criterion for uniform convergence for the case $T=\frac{P_1+...+P_n}{n}$ which is similar to the criterion of \cite{ORobust}[Theorem 3.12] (the constants differ a little):
\begin{corollary}
\label{Quick uniform convergence criterion}
Let $X$ be a Banach space and let $P_1,...,P_n$ be projections in $X$ ($n \geq 2$). Assume that for every $1 \leq j_1 < j_2 \leq n$, there is a projection $P_{j_1,j_2}$ on $\im (P_{j_1}) \cap \im (P_{j_2})$, such that $P_{j_1,j_2} P_{j_1} = P_{j_1,j_2}, P_{j_1,j_2} P_{j_2} = P_{j_1,j_2}$. 

Denote $T=\frac{P_1+...+P_n}{n}$ and assume there is a constant $0 \leq \beta < 1+\frac{1}{n-1}$ such that $\max \lbrace \Vert P_1 \Vert,...,\Vert P_n \Vert \rbrace \leq \beta$ and that $\cos (\angle (P_{j_1},P_{j_2})) <1$ for every $1 \leq j_1, j_2 \leq n$. Then there is $\gamma = \gamma (\beta,n) >0$ such that if 
$$\forall 1 \leq j_1, j_2 \leq n, \cos (\angle (P_1,P_2)) \leq \gamma,$$
then there is a projection $T^\infty$ on $\im (P_1) \cap ... \cap \im (P_n)$ and constants $C=C>0$, $0 < r <1$ that depend only on $\beta$ and $\gamma$ such that $\Vert T^\infty - T^i \Vert \leq C r^{i}$.
\end{corollary} 

\begin{proof}
Fix $\beta < 1+\frac{1}{n-1}$ and define a function 
$$f(x) = \dfrac{n-1}{n} \beta +2 \dfrac{n-1}{n}\dfrac{(x + \beta+\beta^2) x}{1-x}.$$
$f(x)$ is clearly continuous and strictly monotone increasing in the interval $[0,1)$ and $\lim_{x \rightarrow 1} f(x) = \infty$. Also, note that $f(0) <1$, by the choice of $\beta$ and therefore there is $\gamma' >0$ such that $f(\gamma')=1$. Take $\gamma$ to be any number such that $0 < \gamma < \gamma'$ and assume that $ \max_{j_1,j_2} \cos (\angle (P_{j_1},P_{j_2})) \leq \gamma$.
  
By this choice of $\gamma$, we have that
$$r=\min_{1 \leq j \leq n} \left( \dfrac{n-1}{n} \beta +2 \dfrac{n-1}{n} \dfrac{(\cos (\angle (P_j,P_k))+ \beta+\beta^2) \cos (\angle (P_j,P_k))}{1-\cos (\angle (P_j,P_k))} \right) \leq f(\gamma) <1.$$
Therefore we can apply Theorem \ref{averaged projection theorem} with $\alpha_1 =...=\alpha_n=\frac{1}{n}$ and deduce that there is $C=\frac{1+\beta}{1-r}$ such that $\Vert T^\infty - T^i \Vert \leq C r^i$, with $T^\infty$ a projection on $\im (P_1) \cap ... \cap \im (P_n)$. 
\end{proof}

\begin{remark}
\label{gamma, r, C remark}
The constants $C,r,\gamma$ in the above corollary can be computed explicitly: note that finding $\gamma'$ such that $f(\gamma')=1$ boils down to solving a quadratic equation and therefore $\gamma$ can be taken as $\frac{\gamma'}{2}$ and $r$, $C$ can be taken as $r=f(\frac{\gamma'}{2}), C=\frac{1+\beta}{1-r}$.
\end{remark}

Last, we note that $T^i$ converges to a ``canonical'' projection with respect to $P_1,...,P_n$ if such a projection exists.

\begin{proposition}
\label{canonical proposition}
Let $X$ be a Banach space and let $P_1,...,P_n$ be projections in $X$ ($n \geq 2$). Let $\alpha_1,...,\alpha_n$ be positive constants such that $\sum_k \alpha_k =1$. Denote $T = \sum_k \alpha_k P_k$. Assume that $T^i$ converges in the operator norm to $T^\infty$ which is a projection on $\bigcap_{j=1}^n \im (P_j)$. If there is a projection $P_{1,2,...,n}$ on $\bigcap_{j=1}^n \im (P_j)$ such that for every $j$, $P_{1,2,...,n} P_j = P_{1,2,...,n}$, then $T^\infty = P_{1,2,...,n}$.
\end{proposition}

\begin{proof}
Note that for every $i$, we have that $P_{1,...,n} T^i = P_{1,...,n}$ and therefore $T^\infty = P_{1,...,n} T^\infty = P_{1,...,n}$.
\end{proof}

\section{Convergence of random products of projections}
Below, we'll give ``angle criteria'' for the convergence of cyclic, quasi-periodic and random products for projections. Due to the messy nature of the computations, we keep these criteria non-effective, i.e., we do not specify the exact bounds on the angle to ensure uniform convergence, but only show these bounds exist. The interested reader can derive effective versions of all our criteria below by substituting our use of $\gamma$, $r$ and $C$ below by explicit constants such as the ones given in Remark \ref{gamma, r, C remark} and then working out all the computations explicitly with these constants.

Throughout, let $P_1,...,P_n$ be projections in a Banach space $X$. Below we will always assume that for every $1 \leq j_1 < j_2 \leq n$, there is a projection $P_{j_1,j_2}$ on $\im (P_{j_1}) \cap \im (P_{j_2})$, such that $P_{j_1,j_2} P_{j_1} = P_{j_1,j_2}, P_{j_1,j_2} P_{j_2} = P_{j_1,j_2}$.  and therefore $\cos (\angle (P_i,P_j))$ is defined.

\begin{theorem}
\label{Criterion for product to be contracting}
Let $P_1,...,P_n$ be as above. Assume that there is a constant $\beta < 1+\frac{1}{n-1}$ such that   
$$ \max \lbrace \Vert P_1 \Vert,..., \Vert P_n \Vert \rbrace \leq \beta.$$
There is a constant $\gamma = \gamma (\beta,n)$ such that if 
$$ \max \lbrace  \cos(\angle (P_{j_1},P_{j_2})) : 1 \leq j_1 < j_2 \leq n \rbrace \leq \gamma,$$
then there is a projection $T^\infty$ on $\bigcap_{j=1}^n \im (P_j)$ and constants $C>0, 0 < r <1$ that depend only on $\gamma, \beta$ such that for every integer $m \geq n$ and every surjective map $\sigma : \lbrace 1,...,m \rbrace \rightarrow \lbrace 1,...,n \rbrace$ the following holds for every $i$: 
$$\Vert P_{\sigma (m)} ... P_{\sigma (1)}  - T^\infty \Vert \leq \beta^m Cr^{i} +2 \gamma \beta^{m-2} ((2+\beta)^i-1).$$
\end{theorem}

\begin{proof}
Fix $m$, $\sigma$ as above.  Denote $T= \frac{P_1+...+P_n}{n}$ and let $\gamma = \gamma (\beta,n)$ be as in Corollary \ref{Quick uniform convergence criterion}, then by this corollary there are $C>0, 0<r<1$ such that for every $i$,
$$\Vert T^\infty - T^i \Vert \leq  C r^{i},$$
where $T^\infty$ is a projection on $\bigcap_{j=1}^n \im (P_j)$.  
Therefore for every $i$ the following holds:
\begin{align*}
\Vert P_{\sigma (m)} ... P_{\sigma (1)} - T^\infty \Vert = \Vert P_{\sigma (m)} ... P_{\sigma (1)} (I - T^\infty) \Vert \leq \\
\Vert P_{\sigma (m)} ... P_{\sigma (1)} (I- T^i) \Vert + \Vert P_{\sigma (m)} ... P_{\sigma (1)} (T^i - T^\infty) \Vert \leq \\
\Vert P_{\sigma (m)} ... P_{\sigma (1)} (I- T^i) \Vert + \Vert P_{\sigma (m)} ... P_{\sigma (1)} \Vert \Vert T^i - T^\infty \Vert \leq \\
\Vert P_{\sigma (m)} ... P_{\sigma (1)} (I- T^i) \Vert + \beta^m Cr^{i} .
\end{align*}
By the above inequality, it is left to prove that 
$$\Vert P_{\sigma (m)} ... P_{\sigma (1)} (I- T^i) \Vert \leq 2 \gamma \beta^{m-1} (2+\beta)^i.$$
Observe that 
$$I-T^i =I- (I-(I-T))^i = (I-T) \sum_{k=1}^i {i \choose k} (-1)^k (I-T)^{k-1} .$$
Therefore
\begin{align*}
\Vert P_{\sigma (m)} ... P_{\sigma (1)} (I- T^i) \Vert \leq \Vert P_{\sigma (m)} ... P_{\sigma (1)} (I- T) \Vert \left\Vert \sum_{k=1}^i {i \choose k} (-1)^k (I-T)^{k-1} \right\Vert \leq \\ \Vert P_{\sigma (m)} ... P_{\sigma (1)} (I- T) \Vert \left(  \sum_{k=1}^i {i \choose k} (1+\beta)^{k-1} \right) = \dfrac{(2+\beta)^i-1}{1+\beta} \Vert P_{\sigma (m)} ... P_{\sigma (1)} (I- T) \Vert.
\end{align*}
In order to finish the proof, we must show that 
$$\Vert P_{\sigma (m)} ... P_{\sigma (1)} (I- T) \Vert \leq 2 \gamma \beta^{m-2} (m-1) (1+\beta).$$
Note that 
$$\Vert P_{\sigma (m)} ... P_{\sigma (1)} (I- T) \Vert \leq \dfrac{\Vert P_{\sigma (m)} ... P_{\sigma (1)} (I-P_1) \Vert +...+ \Vert P_{\sigma (m)} ... P_{\sigma (1)} (I-P_n) \Vert }{n}$$
and therefore it is enough to show that for every $j$, 
$$\Vert P_{\sigma (m)} ... P_{\sigma (1)} (I-P_j) \Vert \leq 2 \gamma \beta^{m-2} (m-1) (1+\beta) .$$
Let $1 \leq l \leq m$ be the smallest integer such that $\sigma (l)=j$. We will finish the proof by showing by induction that 
$$\Vert P_{\sigma (m)} ... P_{\sigma (1)} (I-P_j) \Vert \leq 2 \gamma \beta^{m-2} (l-1) (1+\beta) .$$
For $l=1$, we have that $\sigma (1)=j$ and therefore 
$$ P_{\sigma (m)} ... P_{\sigma (1)} (I-P_j) =  P_{\sigma (m)} ... P_{\sigma (2)} P_j (I-P_j) =0$$
and we are done. Next, let $l>1$ and assume that the above inequality holds for $l-1$. Define $\sigma' : \lbrace 1,...,m \rbrace \rightarrow \lbrace 1,...,n \rbrace$ as 
$$\sigma' (k) = \begin{cases}
\sigma (k) & k >l \text{ or } k < l-1 \\
\sigma (l) & k = l-1 \\
\sigma (l-1) & k =l
\end{cases}.$$
Note that
\begin{dmath*}
\Vert P_{\sigma (m)} ... P_{\sigma (1)} (I-P_j)  - P_{\sigma' (m)} ... P_{\sigma' (1)} (I-P_j) \Vert = \\
\Vert P_{\sigma (m)} ... P_{\sigma (l+1)} ( P_{\sigma (l)} P_{\sigma (l-1)} -  P_{\sigma (l-1)} P_{\sigma (l)} )  P_{\sigma (l-2)} ... P_{\sigma (1)} (I-P_j) \Vert \leq \\
\Vert P_{\sigma (m)} ... P_{\sigma (l+1)} \Vert \Vert P_{\sigma (l)} P_{\sigma (l-1)} -  P_{\sigma (l-1)} P_{\sigma (l)} \Vert \Vert P_{\sigma (l-2)} ... P_{\sigma (1)} \Vert \Vert I-P_j \Vert \leq \\
\beta^{n-2} (1+\beta) \Vert  P_{\sigma (l)} P_{\sigma (l-1)} -  P_{\sigma (l-1)} P_{\sigma (l)} \Vert \leq \\
\beta^{n-2} (1+\beta) \Vert  P_{\sigma (l)} P_{\sigma (l-1)} -  P_{\sigma (l-1),\sigma (l)} \Vert + \beta^{n-2} (1+\beta) \Vert  P_{\sigma (l-1)} P_{\sigma (l)} -  P_{\sigma (l-1),\sigma (l)} \Vert \leq \\
 2 \gamma \beta^{m-2} (1+\beta),
\end{dmath*}
where the last inequality is due to the assumption that $\cos (\angle (P_{\sigma (l-1)},P_{\sigma (l)} ))\leq \gamma$. 

By definition $\sigma' (l-1) = \sigma (l) = j$ and therefore by the induction assumption, we have that
$$\Vert P_{\sigma' (m)} ... P_{\sigma' (1)} (I-P_j) \Vert \leq 2 \gamma \beta^{n-2} (l-2) (1+\beta) .$$
Combining this inequality with the previous inequality, we get that 
\begin{dmath*}
\Vert P_{\sigma (m)} ... P_{\sigma (1)} (I-P_j) \Vert \leq \\
\Vert P_{\sigma (m)} ... P_{\sigma (1)} (I-P_j) - P_{\sigma' (m)} ... P_{\sigma' (1)} (I-P_j) \Vert + \Vert P_{\sigma' (m)} ... P_{\sigma' (1)} (I-P_j) \Vert \leq \\
2 \gamma \beta^{m-2} (1+\beta) + 2 \gamma \beta^{m-2}(l-2) (1+\beta) = 2 \gamma \beta^{m-2}(l-1) (1+\beta),
\end{dmath*}
as needed.
\end{proof}

\begin{corollary}
\label{main corollary}
Let $P_1,...,P_n$ be as above. For every integer $m \geq n$ and constants $0 \leq q <1$, $0 \leq \beta < 1+ \frac{1}{n-1} $, there is a constant $\gamma = \gamma (\beta,n,m,q)>0$ such that if 
$$ \max \lbrace \Vert P_1 \Vert,..., \Vert P_n \Vert \rbrace \leq \beta \text{ and } \max \lbrace  \cos(\angle (P_{j_1},P_{j_2})) : 1 \leq j_1 < j_2 \leq n \rbrace \leq \gamma,$$
then for every surjective map $\sigma : \lbrace 1,...,m \rbrace \rightarrow \lbrace 1,...,n\rbrace$, we have that 
$$\Vert P_{\sigma (m)} ... P_{\sigma (1)}  - T^\infty \Vert \leq q,$$
where $T^\infty$ as the projection on $\bigcap_{j=1}^n \im (P_j)$ given in the above theorem.
\end{corollary}

\begin{proof}
Fix $m,q,\beta$ as above and fix a surjective map $\sigma : \lbrace 1,...,m \rbrace \rightarrow \lbrace 1,...,n\rbrace$. Let $\gamma_1$ be the bound on the cosine of the angles as in the above theorem and let $r,C$ be the constants corresponding to $\gamma_1$. 

Therefore, if 
$$ \max \lbrace \Vert P_1 \Vert,..., \Vert P_n \Vert \rbrace \leq \beta \text{ and } \max \lbrace  \cos(\angle (P_{j_1},P_{j_2})) : 1 \leq j_1 < j_2 \leq n \rbrace \leq \gamma_1,$$
then by the above theorem we have for every $i$ that
$$\Vert P_{\sigma (m)} ... P_{\sigma (1)}  - T^\infty \Vert \leq \beta^m Cr^{i-1} +2 \gamma_1 \beta^{m-2} ((2+\beta)^i-1).$$

Fix $i_0$ large enough such that $\beta^m Cr^{i_0-1} \leq \frac{q}{2}$. Denote
$$\gamma_2 = \frac{1}{2 \beta^{m-2} ((2+\beta)^{i_0}-1)} \frac{q}{2}.$$
Then for $\gamma = \min \lbrace \gamma_1, \gamma_2 \rbrace$, we have that
$$\Vert P_{\sigma (m)} ... P_{\sigma (1)}  - T^\infty \Vert \leq \beta^m Cr^{i_0-1} +2 \gamma_2 \beta^{m-2} ((2+\beta)^{i_0}-1) \leq \frac{q}{2} + \frac{q}{2} = q.$$
\end{proof}

The above corollary allows us to deduce uniform convergence criteria in the cases of quasi-periodic products (see definition below) and random products under an additional assumption which we call weak consistency:

\begin{definition} [Consistency]
\label{consistency definition}
We shall say that the projections $P_1,...,P_n$ are consistent if for every $\lbrace i_1,...,i_k \rbrace \subseteq \lbrace 1,...,n \rbrace$, there is a projection $P_{i_1,...,i_k}$ on $\im (P_{i_1}) \cap \im (P_{i_1}) \cap ... \cap (P_{i_k})$, such that 
$$\forall 1 \leq j \leq k, P_{i_1,...,i_k} P_{i_j} = P_{i_1,...,i_k}.$$
We shall say that the projections $P_1,...,P_n$ are weakly consistent if there is a projection $P_{1,...,n}$ on $\bigcap_j \im (P_j)$ such that 
$$\forall 1 \leq j \leq n, P_{1,...,n} P_j = P_{1,...,n}.$$
\end{definition}

\begin{example}
In the case where $X= H$ is a Hilbert space and $P_1,...,P_n$ are orthogonal projections on $V_1,...,V_n$ correspondingly, we have that $P_1,...,P_n$ are consistent by taking $P_{i_1,...,i_k}$ to be the orthogonal projection on $V_{i_1} \cap ... \cap V_{i_k}$. Other examples of consistent projections are given in \cite{OBanachCoho}. 
\end{example}

\begin{remark}
We note that by Proposition \ref{canonical proposition}, if $P_1,...,P_n$ are weakly consistent and $T^i=(\frac{P_1+...+P_n}{n})^i$ converges in the operator norm to a projection $T^\infty$ on $\bigcap_j \im (P_j)$, then $T^\infty=P_{1,...,n}$, i.e., 
$$\forall 1 \leq j \leq n, T^\infty P_j = T^\infty.$$
\end{remark}

Next, we'll use Corollary \ref{main corollary} together with the assumption of weak consistency to deduce critria for uniform convergence. We'll start with proving such a criterion for the convergence of $(P_n ... P_1)^i$. We remark that although this case is well studied and various angle considerations were applied to it (see for instance \cite{BadeaGrivauxM}, \cite{PRZ3}, \cite{PRZ}, \cite{PRZ2}), this result is new even in the setting of orthogonal projections in Hilbert spaces.

\begin{proposition}
\label{prop cyclic}
Let $P_1,...,P_n$ be as above and let $0 \leq \beta < 1+\frac{1}{n-1}, 0 < q <1$ be constants. Then there is $\gamma = \gamma (\beta,q)$ such that if 
\begin{enumerate}
\item $ \max \lbrace \Vert P_1 \Vert,..., \Vert P_n \Vert \rbrace \leq \beta \text{ and } \max \lbrace  \cos(\angle (P_{j_1},P_{j_2})) : 1 \leq j_1 < j_2 \leq n \rbrace \leq \gamma.$
\item $P_1,...,P_n$ are weakly consistent.
\end{enumerate}
Then $\Vert (P_n ... P_1)^i - P_{1,...,n} \Vert \leq q^i$.
\end{proposition}

\begin{proof}
Fix $\beta,q$ as above. Take $\gamma = \gamma  (\beta,n,m,q)>0$ given by Corollary \ref{main corollary}. Note that by weak consistency, we have that 
$$P_n ... P_1 - P_{1,...,n} = P_n ... P_1 (I - P_{1,...,n}) = (I-P_{1,...,n} )P_n ... P_1 (I - P_{1,...,n}).$$
Therefore 
$$\Vert (P_n ... P_1)^i - P_{1,...,n} \Vert = \Vert \left( (P_n ... P_1)(I - P_{1,...,n}) \right)^i \Vert \leq q^i,$$
where the last inequality is due to Corollary \ref{main corollary}.
\end{proof}

Next, we'll analyse the case of quasi-periodic products. Recall the following definition: 

\begin{definition}
A map $\tau \in \lbrace 1,...,n \rbrace^{\mathbb{N}}$ will be called quasi-periodic if there is some fixed integer $m \geq n$ (which will be called the quasi-period of $\tau$) such that for every $i \in \mathbb{N}$, 
$$\lbrace \tau (i), \tau (i+1),...,\tau (i+m-1) \rbrace = \lbrace 1,...,n\rbrace.$$
For $\tau$ quasi-periodic, the infinite product $... P_{\tau (i)}...P_{\tau (1)}$ will be called a quasi-periodic product. 
\end{definition}

\begin{theorem}
\label{thm quasi-periodic}
Let $P_1,...,P_n$ be as above and let $0 \leq \beta < 1+\frac{1}{n-1}$. Then for every $m \geq n$, there is $\gamma = \gamma (\beta,n,m)$ such that if 
\begin{enumerate}
\item $ \max \lbrace \Vert P_1 \Vert,..., \Vert P_n \Vert \rbrace \leq \beta \text{ and } \max \lbrace  \cos(\angle (P_{j_1},P_{j_2})) : 1 \leq j_1 < j_2 \leq n \rbrace \leq \gamma.$
\item $P_1,...,P_n$ are weakly consistent.
\end{enumerate}
Then every quasi-periodic product of $P_1,...,P_n$ with a quasi-period $m$ converges uniformly to $P_{1,...,n}$ as the rate of convergence is exponential.
\end{theorem}

\begin{proof}
Fix $\beta$ and $m$ and let $\tau \in  \lbrace 1,...,n \rbrace^{\mathbb{N}}$ be some quasi-periodic map with a quasi-period $m$. Fix some $q<1$ and let $\gamma = \gamma (\beta,n,m,q)$ to be the constant given by Corollary \ref{main corollary}. Then by Corollary \ref{main corollary} and by then quasi-periodicity of $\tau$, if 
$$ \max \lbrace \Vert P_1 \Vert,..., \Vert P_n \Vert \rbrace \leq \beta \text{ and } \max \lbrace  \cos(\angle (P_{j_1},P_{j_2})) : 1 \leq j_1 < j_2 \leq n \rbrace \leq \gamma,$$
then for every $i$, $\Vert P_{\tau (i+m-1)} ... P_{\tau (i)} - P_{1,...,n} \Vert \leq q$.

From the assumption of weak consistency, we have that for every $i_1 < i_2$ the following holds
$$P_{\tau (i_2)} ... P_{\tau (i_1)} (I - P_{1,...,n}) = (I - P_{1,...,n}) P_{\tau (i_2)} ... P_{\tau (i_1)} (I - P_{1,...,n}).$$

Therefore, for every $i$ we have that 
\begin{dmath*}
\Vert P_{\tau (i)} ... P_{\tau (1)} - P_{1,...,n} \Vert = \\
\Vert P_{\tau (i)} ... P_{\tau (\lfloor \frac{i}{m} \rfloor m +1)} \prod_{l=1}^{\lfloor \frac{i}{m} \rfloor} \left( P_{\tau (lm)} ... P_{\tau ((l-1)m+1)} - P_{1,...,n} \right) \Vert \leq \\
 \Vert P_{\tau (i)} ... P_{\tau (\lfloor \frac{i}{m} \rfloor m +1)} \Vert \prod_{l=1}^{\lfloor \frac{i}{m} \rfloor} \Vert P_{\tau (lm)} ... P_{\tau ((l-1)m+1)} - P_{1,...,n} \Vert \leq
 \beta^{m-1} q^{\lfloor \frac{i}{m} \rfloor}.
\end{dmath*}
Therefore $\Vert P_{\tau (i)} ... P_{\tau (1)} - P_{1,...,n} \Vert$ tends to $0$ exponentially fast as $i$ tends to infinity.
\end{proof}

The proof of the analogous theorem when $\tau \in \lbrace 1,...,n \rbrace^{\mathbb{N}}$ is chosen according to some product probability is similar, but requires a little more work:

\begin{theorem}
\label{thm random}
Let $P_1,...,P_n$ be as above and let $0 \leq \beta < 1+\frac{1}{n-1}$. Also, let $\mu$ be a probability measure on $\lbrace 1,...,n\rbrace$ such that for each $j$, $\mu (j) >0$. Then there is $\gamma = \gamma (\beta,n,\mu)$ such that if the following holds:
\begin{enumerate}
\item $ \max \lbrace \Vert P_1 \Vert,..., \Vert P_n \Vert \rbrace \leq \beta \text{ and } \max \lbrace  \cos(\angle (P_{j_1},P_{j_2})) : 1 \leq j_1 < j_2 \leq n \rbrace \leq \gamma.$
\item $P_1,...,P_n$ are weakly consistent.
\end{enumerate}
Then for the space $\lbrace 1,...,n \rbrace^{\mathbb{N}}$ with the infinite product measure $\overline{ \mu} = \prod_{i=1}^\infty \mu$, we have that for $\overline{\mu}$-almost every $\tau \in \lbrace 1,...,n \rbrace^{\mathbb{N}}$, the infinite product $... P_{\tau (i)} ... P_{\tau (1)}$ converges uniformly to $P_{1,...,n}$ and the rate of convergence is exponential.
\end{theorem}

\begin{proof}
Fix $\beta$ and $\mu$ as above. For $k \in \mathbb{N}, \tau \in \lbrace 1,...,n \rbrace^{\mathbb{N}}$, define the sets 
$$A (\tau,k) = \left\lbrace 1 \leq l \leq k : \lbrace \tau(n(l-1)+1), \tau(n(l-1)+2),...,\tau (ln) \rbrace = \lbrace 1,...,n\rbrace \right\rbrace,$$
$$B(\tau,k) = \lbrace 1,...,k \rbrace \setminus A (\tau,k),$$
$$A (\tau ) = \bigcup_{k} A(\tau,k).$$

In a random $n$-tuple of numbers chosen according to the probability $\mu \times \mu \times ... \times \mu$, the probability that each number $1,...,n$ will appear in the $n$-tuple exactly once is $\frac{1}{n! \mu (1) \mu (2) ... \mu (n)}$. Therefore, by the law of large numbers, for $\overline{\mu}$-almost every $\tau \in \lbrace 1,...,n \rbrace^{\mathbb{N}}$ we have that
$$\lim_{k \rightarrow \infty} \dfrac{\vert A(\tau,k) \vert}{k} = \dfrac{1}{n! \mu (1) \mu (2) ... \mu (n)}.$$

Fix $0<\lambda < \frac{1}{n! \mu (1) \mu (2) ... \mu (n)}$. Then, as a result of the law of large numbers mentioned above, for $\overline{\mu}$-almost every $\tau \in \lbrace 1,...,n \rbrace^{\mathbb{N}}$ there is $k_\tau \in \mathbb{N}$ such that for every $k \geq k_\tau$ 
we have that $\frac{\vert A(\tau,k) \vert}{k} \geq \lambda,$ or equivalently, for every $k \geq k_\tau$ we have that $\frac{\vert B(\tau,k) \vert}{k} \leq 1-\lambda$.

Fix $q >0$ such that $\beta^{n(1-\lambda)} q^\lambda <1$ and let $\gamma = \gamma (\beta, n,q)$ given by Corollary \ref{main corollary}. Then by Corollary \ref{main corollary} above, if 
$$\max \lbrace  \cos(\angle (P_{j_1},P_{j_2})) : 1 \leq j_1 < j_2 \leq n \rbrace \leq \gamma,$$
then for every $l \in A (\tau,k)$ we have that 
$$\Vert P_{\tau (ln)} ... P_{\tau ((l-1)n+1)} (I- P_{1,...,n}) \Vert = \Vert P_{\tau (ln)} ... P_{\tau ((l-1)n+1)} - P_{1,...,n} \Vert \leq q.$$

As in the previous proof, from the assumption of weak consistency, we have that for every $i_1 < i_2$ and every $\tau$ the following holds
$$P_{\tau (i_2)} ... P_{\tau (i_1)} (I - P_{1,...,n}) = (I - P_{1,...,n}) P_{\tau (i_2)} ... P_{\tau (i_1)} (I - P_{1,...,n}).$$
For $l \in \mathbb{N}$, define operators $L_l = L_l (\tau)$ as follows
$$L_{l} = \begin{cases}
P_{\tau (ln)} ... P_{\tau ((l-1)n+1)} (I- P_{1,...,n}) & l \in A(\tau) \\
P_{\tau (ln)} ... P_{\tau ((l-1)n+1)} & l \notin A(\tau)
\end{cases}.$$
Then for every $i$, we have (as a result of weak consistency) that 
$$P_{\tau (i)} ... P_{\tau (1)}(I - P_{1,...,n}) = P_{\tau (i)} ... P_{\tau (\lfloor \frac{i}{n} \rfloor n)+1} \left( \prod_{l=1}^{\lfloor \frac{i}{n} \rfloor} L_l \right) (I-P_{1,...,n}).$$
Note that 
$$\forall l \in A(\tau), \Vert L_l \Vert \leq q,$$
$$\forall l \notin A (\tau), \Vert L_l \Vert \leq \beta^n.$$
Therefore for every $i$, we have that
\begin{dmath*}
\Vert P_{\tau (i)} ... P_{\tau (1)}(I - P_{1,...,n}) \Vert \leq \\
\Vert P_{\tau (i)} ... P_{\tau (\lfloor \frac{i}{n} \rfloor n + 1)} \Vert \left( \prod_{l=1}^{\lfloor \frac{i}{n} \rfloor} \Vert L_l \Vert \right) \Vert I-P_{1,...,n} \Vert \leq \\ \beta^{n-1} (\beta^{n})^{\vert B (\tau, \lfloor \frac{i}{n} \rfloor) \vert} q^{\vert A (\tau, \lfloor \frac{i}{n} \rfloor) \vert} \Vert I - P_{1,...,n} \Vert. 
\end{dmath*} 
As a result, by the definition of $k_\tau$, for every $i \geq n k_\tau$ the following holds 
$$\Vert P_{\tau (i)} ... P_{\tau (1)}(I - P_{1,...,n}) \Vert \leq \Vert I - P_{1,...,n} \Vert \beta^{n-1} (\beta^{n(1-\lambda)} q^\lambda)^{\lfloor \frac{i}{n} \rfloor}.$$
Recall that $q$ was chosen to ensure that $\beta^{n(1-\lambda)} q^\lambda <1$ and therefore the right-hand side of the above inequality tends to $0$ (exponentially fast) as $i$ tends to infinity.
\end{proof}

\bibliographystyle{plain}
\bibliography{bibl}

\end{document}